\documentclass[12pt]{article}
\usepackage{amssymb}
\usepackage{graphicx}
\usepackage{amsmath}
\newtheorem{theorem}{Theorem}[section]
\newtheorem{proposition}{Proposition}[section]
\newtheorem{lemma}{Lemma}[section]
\newtheorem{remark}{Remark}[section]

\def \R{\mbox{I\hspace{-.15em}R}}
\def \P{\mbox{I\hspace{-.15em}P}}
\def \E{\mbox{I\hspace{-.15em}E}}

\RequirePackage[latin1]{inputenc} \RequirePackage[T1]{fontenc}
\newenvironment{proof}[1][Proof]{\textbf{#1.} }{\ \rule{0.5em}{0.5em}}

\begin{document}

\author{Auguste Aman\thanks{E-mail address: augusteaman5@yahoo.fr}\\
U.F.R.M.I, Université de Cocody,\\
22 BP 582 Abidjan 22, Côte d'Ivoire}
\date{}
\title{Homeomorphism of solutions to backward doubly SDEs and applications\thanks{This works is partially support by Fellowship grant of AMMSI}}
\maketitle
\begin{abstract}
In this paper we study the homeomorphic properties of the solutions to one dimensional backward doubly
stochastic differential equations under suitable assumptions, where the terminal values depend on a real
parameter. Then, we apply them to the solutions for a class of second order quasilinear parabolic stochastic partial
differential equations.
\end{abstract}

\section{Introduction and Main results}
\setcounter{theorem}{0} \setcounter{equation}{0}
For fix a positive real number
$T>0$, let $\{W_{t}, 0\leq t\leq T\}$ and $\{B_{t},\ 0\leq
t\leq T\}$ are two mutually independent standard Brownian motions
with values in $\mbox{I\hspace{-.15em}R}$, defined respectively on the two
probability spaces $(\Omega_{1},\mathcal{F}_{1},{\P}_{1})$ and
$(\Omega_{2},\mathcal{F}_{2},{\P}_{2})$. Next we consider $(\Omega,\mathcal{F},\P)$ the probability space defined by
\begin{eqnarray*}
\Omega=\Omega_{1}\times\Omega_{2},\,\,\mathcal{F}=\mathcal{F}_{1}\otimes\mathcal{F}_{2}\,\,
\mbox{ and}\,\, \P=\P_{1}\otimes\P_{2},
\end{eqnarray*}
and $\mathcal{N}$ denote the class of $\P$-null sets of $\mathcal{F}$. For each $t\in[0, T]$,
we define
\begin{eqnarray*}
\mathcal{F}_{t}=\mathcal{F}_{t}^{W}\otimes\mathcal{F}^{B}_{t,T}
\end{eqnarray*}
where for any process $\{\eta_t\},\, \mathcal{F}^{\eta}_{s,t}=\sigma\{\eta_{r}-\eta_{s},s\leq r\leq t\}\vee
\mathcal{N},\; \mathcal{F}^{\eta}_{t}=\mathcal{F}^{\eta}_{0,t}$.

Let us remark that the collection ${\bf F}= \{\mathcal{F}_{t},\ t\in
[0,T]\}$ is neither increasing nor decreasing and it does not
constitute a filtration. Further, we assume that, random variables,
$\zeta(\omega_{2}),\; \omega_{2}\in \Omega_{2}$ are considered as
random variables on $\Omega $ via the following identification:
\begin{eqnarray*}
\zeta(\omega_{1},\omega_{2})=\zeta(\omega_{2}).
\end{eqnarray*}

Now, let consider the following one dimensional backward doubly stochastic differential equation (BDSDE):
\begin{eqnarray}
Y_{t}^{\xi}=\xi+\int_{t}^{T}f(s,Y_{s}^{\xi},Z_{s}^{\xi})\,
ds+\int_{t}^{T}g(s,Y_{s}^{\xi},Z_s^{\xi})d\overleftarrow{B_{s}} -\int_{t}^{T}Z_{s}^{\xi}dW_{s},\;\; t\in[0,T],\label{BDSDE}
\end{eqnarray}
where the terminal condition $\xi\in L^2(\Omega,\mathcal{F}_{T}, \P)$, $f(s,\omega, y, z): [0, T ] \times\Omega\times\R\times \R\rightarrow\R$
and $\ g(s,\omega_2, y, z) : [0, T ] \times\Omega_2\times\R\times \R\rightarrow\R$ satisfy that

\begin{description}
\item $({\bf H}_{f}^{1})$\ for all $y, z\in\R$, the process $t\mapsto (f(t, y, z), g(t, y, z))$ is $\mathcal{F}_t$-adapted,\\
$\displaystyle{\int^{T}_{0}|f (s, 0, 0)|ds\in L^2(\Omega,\mathcal{F}_{T}, \P)}$,\\
for some $C_{f}> 0$\ and all $(s,\omega_1,\omega_2)\in[0, T ]\times\Omega_1\times\Omega_2, y, y', z, z'\in\R,\\\\
\begin{array}{l}
|f(s,\omega,y,z)-f(s',\omega,y',z')|\leq C_{f}\left(|y-y'|+|z-z'|\right).
\end{array}
$
\item $({\bf H}_{g}^{1})$\ for all $y, z\in\R$, the process $t\mapsto g(t, y, z)$ is $\mathcal{F}_t$-adapted,\\
$\displaystyle{\int^{T}_{0}|g(s, 0, 0)|^{2}ds\in L^1(\Omega_2,\mathcal{F}_{T}; \P)}$,\\
for some $C_g> 0,\ 0<\alpha_{g}<1$ and all $(s,\omega_2)\in[0, T ]\times\Omega_2, y, y', z, z'\in\R$,\\\\
$
\begin{array}{l}
|g(s,\omega_2,y,z)-g(s,\omega_2,y',z')|^{2} \leq  C_{g}|y-y'|^2+\alpha_{g}|z-z'|^2.
\end{array}
$
\end{description}

Solving such an equation is to find a pair of $\mathcal{F}_t$-adapted processes $(Y_t,Z_t)$ such that Eq.$(\ref{BDSDE})$ holds. Applying the extension of representation theorem of martingales with
respect to $\mathcal{G}_t$, denoting by $\mathcal{G}_t=\mathcal{F}_{t}^{W}\vee\mathcal{F}^{B}_T$, the existence and uniqueness for Eq.$(\ref{BDSDE})$ were first established by Pardoux-Peng \cite{PP}. Since then, the theory for BDSDEs achieved fruitful results, which has
also been proved to be an efficient tool such as probabilistic interpretation of stochastic partial differential equations. We can cite the two remarkable works due to Bukdhan and Ma (cf. \cite{BM1, BM2}), which introduced the very interest notion of stochastic viscosity solution of SPDE.

In this paper, we consider the following problem: if the terminal condition $\xi$ is replaced
by a family of $\mathcal{F}_T$-measurable random variables $\xi(x)$ depending on a parameter $x\in\R$ and such that $x\mapsto \xi(x,\omega)$ a.s. are homeomorphisms on $\R$, could the corresponding solution to Eq.$(\ref{BDSDE})$ $x\mapsto Y^{\xi(x)}_t$ be homeomorphisms on $\R$? When all the things are non-random,
this problem is of course affirmative. In the case of forward stochastic differential equations, stochastic homeomorphisms flows are well known and were studied in \cite{K,P,YO,RZ}, etc. In particular, in his book \cite{P}, Protter studied the more general stochastic flows of SDEs driven
by semimartingales. The situation for BSDEs have been investigated recently by Huijie and Zhang \cite{QZ}. The proof is based on given extended comparison theorem for
BSDE, which is used to compare the solutions of the BSDE with the backward ordinary differential equation, and on Yamada-Ogura's argument \cite{YO}.
The aim of this paper is to adapted the same step to the BDSDE in order to answer the above question positively. Here, we use the new version of comparison theorem for BDSDE to compare the solutions of the above BDSDE with the backward stochastic differential equation.

We are mainly devoted to proving the following two
results.
\begin{theorem}
In addition to $({\bf H}^{1}_{f})$ and $({\bf H}^{1}_{g})$, we also assume that
\begin{description}
\item $({\bf H}^{2}_{f})$\ the random variable $\int^{T}_{0}|f (s, 0, 0)|ds$ is bounded by $C_0$;
\item $({\bf H}^{2}_{g})$\ the function $g$ is linear in $y$ and independent of $z$ i.e there exist a functions $a:[0,T]\times\Omega_2\rightarrow \R$ such that
$
\begin{array}{l}
g(s,\omega_2,y)=a(s,\omega_2)y
\end{array}
$
verifying $|a(s,\omega_2)|<C_g/2$ a.s. \, for all\; $s\in[0,T]$
\item $({\bf H}^{1}_{\xi})$\ for almost all $\omega,\ x\mapsto\xi(\omega,x)$ is increasing (or decreasing) and a homeomorphism on $\R$;
\item $({\bf H}^{2}_{\xi})$\ for any $R > 0$, there are $\delta_R,\ C_R > 0$ such that $\E|\xi(x) - \xi(y)|^2\leq C_R|x - y|^{1+\delta_R}$ for all
$|x|, |y| \leq R$;
\item $({\bf H}^{3}_{\xi})$\ for some $R_0 > 0$ and $\varepsilon > 0, \inf_{|x|\geq R_0} \xi(x,\omega)/h(x)> \varepsilon$ a.s., where $h(x)$ is a real continuous function on $\R$ satisfying $\lim_{x\rightarrow\pm\infty}h(x)=\pm\infty$ (or $\lim_{x\rightarrow\pm\infty}h(x)=\mp\infty$).
\end{description}
Then, for almost all $\omega\in\Omega$, the map $\R\ni x \rightarrow Y_t^{\xi(x)}(\omega)\in\R$ is a homeomorphism for every
$t\in[0, T ]$.
\end{theorem}

The proof of this theorem is based on an extended comparison theorem for
BDSDE given in Section 2, which is used to compare the solutions of the
BDSDE with the backward stochastic differential equation

\begin{theorem}
In addition to $({\bf H}^{1}_{f})$, $({\bf H}^{2}_{g})$, $({\bf H}^{1}_{\xi})$ and $({\bf H}^{2}_{\xi})$, we assume that
\begin{description}
\item $({\bf H}^{2'}_{f})$\ for some $C_1>0$, and $\varepsilon_1>0$, it holds that\\
$
\begin{array}{l}
y.f (s, \omega,y,z)\geq -C_1|z|^{2},\ \mbox{for all}\; (s,\omega)\in[0, T ]\times\Omega\, \mbox{and}\; |y|\leq \varepsilon_1,\ z\in\R;
\end{array}
$
\\
\item $({\bf H}^{3'}_{\xi})$\ for some $\beta <\frac{1-2C1}{2}\wedge 0,\ \liminf_{|x|\rightarrow\infty}\E|\xi(x)|^{4\beta} = 0$.
\end{description}
Then, for almost all $\omega\in\Omega$, the map $\R\ni x \rightarrow Y_t^{\xi(x)}(\omega)\in\R$ is a homeomorphism for every
$t\in[0, T ]$.
\end{theorem}

The proof of this theorem is based on Yamada-Ogura's argument [10]. An
elementary function satisfying $({\bf H}^{1}_{f})$ and $({\bf H}^{2'}_{f})$ is $f (y,z) = y +\arctan y.(1+\sin z)$. Moreover,
it is clear that $({\bf H}^{2}_{g})$ and $({\bf H}^{3}_{\xi})$ implies respectively $({\bf H}^{1}_{g})$ and $({\bf H}^{3'}_{\xi})$. These two theorems will be proved in Section 2.

A simple financial meaning for these results is explained as follows: if one investor wants
to get sufficiently high return at a future time, then he or she must invest enough money at the
present time.

In Section 3, we apply Theorem 1.2 to the following backward doubly stochastic differential equation
coupled with a forward stochastic differential equation:
\begin{eqnarray}
\left\{
\begin{array}{l}
\displaystyle{X_{s}^{t,x}=x+\int_{t}^{s}b(X_{r}^{t,x})dr+\int_{t}^{s}\sigma(X^{t,x}_{r})dW_{r}}\\
\displaystyle{Y_{s}=h(X^{t,x}_{T})+\int_{s}^{T}f(r,X^{t,x}_{r},Y_{r}^{t,x},Z_{r}^{t,x})\,
dr+\int_{s}^{T}a(r,X_{r}^{t,x})Y_{r}^{t,x}d\overleftarrow{B_{r}}}\\
\displaystyle{-\int_{s}^{T}Z^{t,x}_{r}dW_{r}},\;\;\; s\in[t,T],\label{FBSDE}
\end{array}\right.
\end{eqnarray}
where $b, \sigma, h : \R \rightarrow\R$, $f: [0, T ]\times\R\times\R\times\R\rightarrow\R$ and $a,c:[0,T]\times\R\rightarrow\R$ are Borel measurable functions. This
type equation in general case was proved in Pardoux-Peng \cite{PP} to be related to some second order quasilinear
parabolic stochastic partial differential equations under some regularity assumptions on the above functions. Our another aim in the present paper is to obtain the homeomorphic property for $x\mapsto Y^{ t,x}_s$, and furthermore, get the homeomorphic property for the solutions to some second order parabolic partial differential equations.

Throughout the paper, $C$ with or without indices will denote different positive constants (depending on the indices) whose values are not important.
\section{Proofs of Main results}
\setcounter{theorem}{0} \setcounter{equation}{0}
Before proving our main results, let us first prove a other version for the comparison
theorem of BDSDEs which need a slight constraint on the coefficient $g$, that is $g$ no dependent of $z$ . Here the method is borrowed from El Karoui et al. \cite{Kal} .
\begin{theorem}
\begin{description}
\item $(i)$\ $f^2$ satisfies $({\bf H}_{f^1}^{1})$ with Lipschitz constant $C_{f^1}$, and $f^{1}(s,\omega, y, z)\geq f^{2}(s,\omega, y, z)$ for all
$(s, \omega)\in[0, T ]\times\Omega$ and $y, z, \in\R$;
\item $(ii)$,\ $V^{1}_t$ and $V^{2}_t$ are $\mathcal{F}_t$-adapted and finite variation processes with $d(V^{1}-V^{2})^{+}_t\leq d\beta_t$ for some
determined and increasing function $\beta_t$, where $d(V^1-V^2)^{+}_t$ denotes the positive variational;
\item $(iii)$\ $\xi^1,\ \xi^2, V^1_T, V^2_T \in L^2(\Omega,\mathcal{F}_T, \P),\; (\xi^1 - \xi^2 + V^1_T - V^2_T)> \varepsilon$ a.s. for some $\varepsilon > 0$.
Let $Y^1_t$ and $Y^2_t$ be the solutions to the following BDSDEs:
\begin{eqnarray*}
Y^1_t = \xi^{1} + V^{1}_t +\int_{t}^{T}f^{1}(s, Y^{1}_s, Z^{1}_s)ds +\int_{t}^{T}g(s, Y^{1}_s)d\overleftarrow{B}_s-\int_t^T
Z^{1}_sdW_s ,\  t\in[0, T],\\
Y^1_t = \xi^{2} + V^{2}_t +\int_{t}^{T}f^{2}(s, Y^{2}_s, Z^{2}_s)ds +\int_{t}^{T}g(s, Y^{2}_s)d\overleftarrow{B}_s-\int_t^T
Z^{2}_sdW_s ,\  t\in[0, T].
\end{eqnarray*}
Then we have for any $t\in [0, T ]$
\begin{eqnarray*}
Y^1_t - Y^2_t > e^{-(C_{f^{1}}+C_g) T}\varepsilon- e^{(C_{f^{1}}+C_g)T }, \:\:\mbox{a.s}.
\end{eqnarray*}
In particular, if $V^{1}_t = V^{2}_t$ and $\xi^1 > \xi^2$ a.s., then for all $t\in[0, T]$
\begin{eqnarray*}
Y^1_t > Y^2_t,\:\:\:\mbox{a.s}.
\end{eqnarray*}
\end{description}
\end{theorem}
\begin{proof}
Put
\begin{eqnarray*}
\hat{\xi} &=& \xi^{1}-\xi^{2},\;\;\;\; \hat{Y}_t = Y_t^{1}-Y_t^{2},\;\;\;\;\; \hat{V}_t = V^{1}_t-V^{2}_t,\;\;\;\;\;\; \hat{Z}_t = Z^{1}_t-Z^{2}_t\\
\hat{f}_t &=&f^1(t, Y^ 2_t , Z^2_t )-f^{2}(t, Y^{2}_t, Z^{2}_t),
\end{eqnarray*}
and
\begin{eqnarray*}
a_t &=& [f^1(t, Y^1_t,Z^1_t) - f^1(t, Y^2_t,Z^1_t)]/(Y^1_t-Y^2_t){\bf 1}_{\{Y^1_t\neq Y^2_t\}},\\
b_t &=& [f^1(t, Y^2_t,Z^1_t)-f^1(t,Y^2_t, Z^2_t)]/(Z^1_t-Z^2_t){\bf 1}_{\{Z^1_t\neq Z^2_t\}},\\
c_t &=& [g(t, Y^1_t) - g(t, Y^2_t)]/(Y^1_t-Y^2_t){\bf 1}_{\{Y^1_t\neq Y^2_t\}}.
\end{eqnarray*}
Then for fixed $r\in[0, t]$
\begin{eqnarray*}
\hat{Y}_t &=& \hat{\xi}+ \hat{V}_t +\int_{t}^{T}[a_s\hat{Y}_s+b_s\hat{Z}_s+\hat{f}_s]ds+\int_{t}^{T}c_s\hat{Y}_sdB_s-\int_t^T
\hat{Z}_sdW_s\\
&=&\hat{Y}_r+ \hat{V}_t-\hat{V}_r -\int_{r}^{t}[a_s\hat{Y}_s+b_s\hat{Z}_s+\hat{f}_s]ds-\int_{r}^{t}c_s\hat{Y}_sdB_s+\int_r^t
\hat{Z}_sdW_s,
\end{eqnarray*}
is a linear BDSDE. It is well not that this equation has an explicit solution given by:
\begin{eqnarray}
\hat{Y}_t=Q^t_T\hat{Y}_T-\int_{t}^{T}Q^t_sd\hat{V}_s+\int_{t}^{T}Q^t_s\hat{f}_sds-\int_{t}^{T}Q^t_s(b_s\hat{Y}_s+b_s\hat{Z}_s+\hat{Z}_s)dW_s,\label{linear}
\end{eqnarray}
where for fixed $t\in[0,T]$ we define
\begin{eqnarray*}
Q^t_s=\exp\left(\int_{t}^{s}b_rdW_r-\frac{1}{2}\int_{t}^{s}|b_r|^{2}dr+\int_{t}^{s}c_rdB_r-\frac{1}{2}\int_{t}^{s}|c_r|^{2}dr+\int_{t}^{s}a_sds\right),
\\ \, t\leq s\leq T.
\end{eqnarray*}
Noting that
\begin{eqnarray*}
|a_t |\leq C_{f^1},\;\;\; |b_t|\leq C_{f^1},\;\;\; |c_t|\leq C_{g},
\end{eqnarray*}
$Q^t_s$ is well defined. Moreover it is clear that for any $0\leq t \leq s\leq T$, we have
\begin{eqnarray*}
e^{-(C_{f^{1}}+C_g)(s-t)}\leq \E[Q^t_s|\mathcal{F}_t]\leq e^{(C_{f^{1}}+C_g)(s-t)}
\end{eqnarray*}
Taking expectation with respect to $\mathcal{F}_t$ in $(\ref{linear})$, we obtain
\begin{eqnarray*}
\hat{Y}_t=\E[Q^t_T\hat{Y}_T|\mathcal{F}_t]-\int_{t}^{T}\E[Q^t_s|\mathcal{F}_t]d\hat{V}_s+\int_{t}^{T}\E[Q^t_s\hat{f}_s|\mathcal{F}_t]ds.
\end{eqnarray*}
Therefore, by $(i)$, $(ii)$ and $(iii)$ we have
\begin{eqnarray*}
\hat{Y}_t&=&\E[Q^t_T\hat{Y}_T|\mathcal{F}_t]-\int_{t}^{T}\E[Q^t_s|\mathcal{F}_t]d\hat{V}_s+\int_{t}^{T}\E[Q^t_s\hat{f}_s|\mathcal{F}_t]ds\\
&>&\varepsilon\E[Q^t_T\hat{Y}_T|\mathcal{F}_t]-\int_{t}^{T}\E[Q^t_s|\mathcal{F}_t]d\beta_s\\
&\geq&e^{-(C_{f^{1}}+C_g+\alpha_g)T}\varepsilon -e^{(C_{f^{1}}+C_g+\alpha_g)T}\beta_T, \;\;\; a.s.
\end{eqnarray*}
The proof is thus complete.
\end{proof}

For simplicity of the notation, we write $Y^{x}_t=Y^{\xi(x)}_t$ in the sequel of the paper. Let us prove a useful Lemma.
\begin{lemma}
Assume $({\bf H}^{1}_{f})$, $({\bf H}^{2}_{g})$ and $({\bf H}^{2}_{\xi})$ hold. Then for any $R>0$ we have
\begin{eqnarray*}
\E\left[\sup_{0\leq t\leq T}|Y^{x}_t-Y^{y}_t|^{2}\right]\leq C_{R}|x-y|^{1+\delta_{R}},\;\;\;\; |x|,\ |y|\leq R.
\end{eqnarray*}
In particular; $\{Y^{x}_t: (t,x)\in[0,T]\times\R$ admits a bicontinuous modification. If in addition $({\bf H}^{1}_{\xi})$ holds, then
\begin{eqnarray*}
\P(\omega:\, Y^{x}_t(\omega)<Y^{y}_t(\omega),\, \forall\  x<y, \,\; t\in[0,T])=1.
\end{eqnarray*}
\end{lemma}
\begin{proof}
Set $\bar{Y}_t=Y^{x}_t-Y^{y}_t$ and $\bar{Z}_t=Z^{x}_t-Z^{y}_t$. By Itô's formula, we have
\begin{eqnarray*}
|\bar{Y}_t|^{2}+\int^{T}_{t}|\bar{Z}_s|^{2}ds&=&|\bar{Y}_T|^2+2\int_{t}^{T}\bar{Y}_s[f(s,Y^{x}_s,Z^{x}_s)-f(s,Y^{y}_s,Z^{y}_s)]ds\\
&&+2\int_{t}^{T}|g(s,Y^{x}_s,Z^{x}_s)-g(s,Y^{y}_s,Z^{y}_s)|^2ds\\
&&+2\int_{t}^{T}\bar{Y}_s[g(s,Y^{x}_s,Z^{x}_s)-g(s,Y^{y}_s,Z^{y}_s)]dB_s
-2\int_{t}^{T}\bar{Y}_s\bar{Z}_sdW_s.
\end{eqnarray*}
Taking expectation and using $({\bf H}^{1}_{f})$, $({\bf H}^{1}_{g})$   and Young's inequality, we deduce that
\begin{eqnarray*}
\E|\bar{Y}_t|^{2}+(1-\alpha-\gamma)\E\int^{T}_{t}|\bar{Z}_s|^{2}ds&\leq &\E|\bar{Y}_T|^2+C\E\int_{t}^{T}|\bar{Y}_s|^{2}ds,
\end{eqnarray*}
for $\gamma$ taken small enough such that $1-\alpha-\gamma>0$. It then follows from Gronwall's inequality and $({\bf H}^{1}_{\xi})$ that for any $t\in[0,T]$
\begin{eqnarray}
\E|\bar{Y}_t|^{2}+(1-\alpha-\gamma)\E\int^{T}_{t}|\bar{Z}_s|^{2}ds&\leq &C\E|\bar{Y}_T|^2\leq C|x-y|^{1+\delta_{R}}.\label{esti2}
\end{eqnarray}
Hence, by Burkölder's inequality we have
\begin{eqnarray*}
\E\left[\sup_{0\leq t\leq T}|\bar{Y}_t|^{2}\right]&\leq& \E|\bar{Y}_T|^2+C\int_{0}^{T}\E|f(s,Y^{x}_s,Z^{x}_s)-f(s,Y^{y}_s,Z^{y}_s)|^{2}ds\\
&&+2\int_{0}^{T}\E|g(s,Y^{x}_s,Z^{x}_s)-g(s,Y^{y}_s,Z^{y}_s)|^2ds\\
&&+C\E\left[\sup_{0\leq t\leq T}\left|\int_{t}^{T}[g(s,Y^{x}_s,Z^{x}_s)-g(s,Y^{y}_s,Z^{y}_s)]dB_s\right|^{2}\right]\\
&&+C\E\left[\sup_{0\leq t\leq T}\left|\int_{t}^{T}\bar{Z}_sdW_s\right|^{2}\right]\\
&\leq&C\E\left\{|\bar{Y}_T|^2+\int_{0}^{T}|\bar{Y}_s|^{2}ds+\int_{0}^{T}|\bar{Z}_s|^{2}ds\right\}\\
&\leq& C|x-y|^{1+\delta_{R}}.
\end{eqnarray*}
The proof is finished.
\end{proof}

We now give the proof of Theorem 1.1.

{\bf Proof of Theorem 1.1.} We assume that here $g$ is linear on $y$ and and independent of $z$ i.e  $({\bf H}^{2}_{g})$ holds. By $({\bf H}^{1}_{f})$, $({\bf H}^{1}_{g})$, Theorem 2.1 and Lemma 2.2, we know that $x\mapsto Y^{x}_t(\omega)$ are continuous injective for all $t\in[0,T]$, a.s. Next we prove the onto property of $x\mapsto Y^{x}_t(\omega)$. Let $(\hat{Y}^{x}_t,\hat{Z}^{x}_t)$ and $(\tilde{Y}^{x}_t,\tilde{Z}^{x}_t)$ be respectively the solutions to equations:
\begin{eqnarray*}
\hat{Y}^{x}_t&=&\xi(x)+\int^T_t\left[|f(s,0,0)|+C_{f}(|\hat{Y}^{x}_s|+|\hat{Z}^{x}_s|)\right]ds\\
&&+\int^{T}_{t}a(s)\hat{Y}^x_sd\overleftarrow{B}_s-\int^T_t\hat{Z}^{x}_sdW_s,
\end{eqnarray*}
and
\begin{eqnarray*}
\tilde{Y}^{x}_t&=&\xi(x)-\int^T_t\left[|f(s,0,0)|+C_{f}(|\tilde{Y}^{x}_s|+|\tilde{Z}^{x}_s|)\right]ds\\
&&+\int^{T}_{t}a(s)\tilde{Y}^x_sd\overleftarrow{B}_s-\int^T_t\tilde{Z}^{x}_sdW_s,
\end{eqnarray*}
where $C_{f}$ is the Lipschitz constant of $f$.

Once again appying the comparison theorem of BDSDE, we obtain
\begin{eqnarray}
Y^{x}_t\leq \hat{Y}^{x}_t, \;\;\; \forall\ x\in\R, \, \forall\ t\in[0,T],\;\; \mbox{a.s.},\label{Comparison1}\\\nonumber\\
\tilde{Y}^{x}_t\leq Y^{x}_t, \;\;\; \forall\ x\in\R, \, \forall\ t\in[0,T],\;\; \mbox{a.s.}\label{Comparison2}
\end{eqnarray}
For $0<\varepsilon_0<\varepsilon$, choosing $M>R_0$ sufficiently large such that
\begin{eqnarray*}
|h(x)|\geq \frac{C_0 T\ e^{2C_0 T}}{\varepsilon-\varepsilon_0},\;\;\; \mbox{for all}\, |x|>M,
\end{eqnarray*}
where $C_{0}$ is the constant in $({\bf H}^{2}_{f})$.

Then by $({\bf H}^{3}_{\xi})$ we have \\\\
$
\begin{array}{l}
\xi(x)+C_{0}Te^{2C_0 T}\leq h(x)\varepsilon_0,\:\:\: \forall\ x<-M\,\,\, \mbox{a.s.}\\\\
\xi(x)-C_{0}Te^{2C_0 T}\geq h(x)\varepsilon_0,\:\:\: \forall\ x>M\;\; \mbox{a.s.}
\end{array}
$
\\\\
Set $
\begin{array}{l}
X^{\pm}_t(x)= h(x)\varepsilon_0\exp\left(\pm C_{f}(T-t)+\int^{T}_{t}a(s)d\overleftarrow{B}_s-\frac{1}{2}\int^{T}_{t}|a(s)|^{2}ds\right).
\end{array}
$
Then
\begin{eqnarray}
X^{\pm}_t(x)=h(x)\varepsilon_0\pm C_{f}\int^{T}_{t}X^{\pm}_s(x)ds+\int^{T}_{t}a(s)X^{\pm}_s(x)d\overleftarrow{B}_s.\label{SDE}
\end{eqnarray}
By $({\bf H}^{2}_{f})$ and Theorem 2.1 we have
\begin{eqnarray}
\hat{Y}^{x}_t\leq X^{+}_t(x)= h(x)\varepsilon_0.\exp\left(C_{f}(T-t)+\int^{T}_{t}a(s)d\overleftarrow{B}_s-\frac{1}{2}\int^{T}_{t}|a(s)|^{2}ds\right),\nonumber\\ \forall\ x<-M\,\,\, \mbox{a.s.}\label{C1}\\\nonumber\\
h(x)\varepsilon_0.\exp\left(-C_{f}(T-t)+\int^{T}_{t}a(s)d\overleftarrow{B}_s-\frac{1}{2}\int^{T}_{t}|a(s)|^{2}ds\right)=X^{-}_t(x)\leq \tilde{Y}^{x}_t,\nonumber\\\;\; \forall\ x>M\,\,\, \mbox{a.s.}\label{C2}
\end{eqnarray}
Thus, we finally get from $(\ref{Comparison1})$ to $(\ref{C2})$ and $({\bf H}^{3}_{\xi})$
\begin{eqnarray*}
\lim_{\uparrow \infty}Y^{x}_{t}=+\infty,\;\;\;\;\;\;\;\lim_{\downarrow -\infty}Y^{x}_{t}=-\infty,\;\;\; \forall\ t\in[0,T]\;\;\; \mbox{a.s.},
\end{eqnarray*}
which complete the proof of surjection of the mapping $x\mapsto Y^{x}_t(\omega)$.
$\blacksquare$

The following lemma plays a crucial role for proving Theorem 1.2.
\begin{lemma}
Assume $({\bf H}^{1}_{f})$, $({\bf H}^{1}_{g})$, $({\bf H}^{2'}_{f})$ and $({\bf H}^{3'}_{\xi})$. Moreover, we suppose that $g(s,\omega,y,z)=0$ a.s. Then
\begin{eqnarray*}
\liminf_{|x|\rightarrow +\infty}\E\left(\sup_{0\leq t\leq T}|Y^{x}_t|^{4\beta}\right)=0,
\end{eqnarray*}
where $\beta$ is given in $({\bf H}^{3'}_{\xi})$.
\end{lemma}
\begin{remark}
We observe that if $({\bf H}^{2}_{g})$ holds, hence $g(s,\omega,y,z)=0$ and $({\bf H}^{1}_{g})$ is verified.
\end{remark}
\begin{proof}
In the following proof, by drawing the sequence if necessary, without any loss of
generality we may assume that for all $x\in\R$\:
$
\begin{array}{c}
\E|\xi(x)|^{4\beta}<\infty.
\end{array}
$

For any $\varepsilon > 0$, by Itô's formula we have
\begin{eqnarray}
(|Y^{x}_t|^2+\varepsilon)^{\beta}&=&(|\xi(x)|^2+\varepsilon)^{\beta}+2\beta\int_{t}^T(|Y^{x}_s|^2+\varepsilon)^{\beta-1}Y^{x}_sf(s,Y^{x}_s,Z^{x}_s)ds\nonumber\\
&&-2\beta\int_{t}^{T}(|Y^{x}_{s}|^{2}+\varepsilon)^{\beta-1}Y^{x}_{s}Z^{x}_{s}dW_s\nonumber\\
&&+2\beta\int_{t}^{T}(|Y^{x}_{s}|^{2}+\varepsilon)^{\beta-1}Y^{x}_{s}g(s,Y^{x}_s,Z^{x}_{s})d\overleftarrow{B}_s\nonumber\\
&&-2\beta(\beta-1)\int_{t}^{T}(|Y^{x}_{s}|^{2}+\varepsilon)^{\beta-2}|Y^{x}_{s}|^{2}|Z^{x}_{s}|^{2}ds\nonumber\\
&&-\beta\int_{t}^{T}(|Y^{x}_{s}|^{2}+\varepsilon)^{\beta-1}|Z^{x}_{s}|^{2}ds\nonumber\\
&&+2\beta(\beta-1)\int_{t}^{T}(|Y^{x}_{s}|^{2}+\varepsilon)^{\beta-2}|Y^{x}_{s}|^{2}|g(s,Y^{x}_s,Z^{x}_{s})|^{2}ds\nonumber\\
&&+\beta\int_{t}^{T}(|Y^{x}_{s}|^{2}+\varepsilon)^{\beta-1}|g(s,Y^{x}_s,Z^{x}_{s})|^{2}ds.\nonumber\\\label{esti1}
\end{eqnarray}
Let us first prove the a priori estimate
\begin{eqnarray}
\sup_{0\leq t\leq T}\E(|Y^{x}_t|^{2})^{2\beta}.\label{esti3}
\end{eqnarray}
It is clear that $({\bf H}^{1}_{f})$ and $({\bf H}^{2'}_{f})$ implies $f(s,0,0)=0$, which together with $({\bf H}^{1}_{f})$ then gives
\begin{eqnarray*}
|f(s,y,z)|\leq C_{f}(|y|+|z|).
\end{eqnarray*}
Since $\beta<\frac{1-2C_1}{2}\wedge 0$, we can choose $\delta>0$ such that
$
\begin{array}{l}
[\beta(C_1+\delta C_{f})+8\beta^2-2\beta-\alpha_{g}(16\beta^2-4\beta)]>0.
\end{array}
$
Thus replacing $\beta$ by $2\beta$ in the above estimates, taking expectation and letting $\varepsilon\downarrow 0$, we have by $ab\leq \delta +b^{2}/(4\delta)$ and monotonic convergence theorem:
\begin{eqnarray*}
&&\E|Y^{x}_t|^{4\beta}+\left[\beta(C_1+\delta C_{f})+2\beta(4\beta-1)-2\alpha_{g}\beta(4\beta-1)\right]\int_{t}^T\E(|Y^{x}_s|^{2(2\beta-1)}|Z^{x}_s|^{2})ds\\
&&\leq\E|\xi(x)|^{4\beta}+\left[4|\beta|C_{f}\left(1+\frac{1}{4\delta}\right)+2C_g\beta(4\beta-1)\right]\int_{t}^T\E|Y^{x}_s|^{4\beta}ds.
\end{eqnarray*}
Hence, Gronwall's inequality gives for any $t\in[0, T]$
\begin{eqnarray*}
&&\E|Y^{x}_t|^{4\beta}+\int_{t}^T\E(|Y^{x}_s|^{2(2\beta-1)}|Z^{x}_s|^{2})ds\leq C\E|\xi(x)|^{4\beta}.
\end{eqnarray*}
Thus $(\ref{esti3})$ follows by $(\ref{esti2})$ .

Similar to the above calculations, from $(\ref{esti1})$ we may derive that
\begin{eqnarray*}
|Y^{x}_t|^{2\beta}&\leq& |\xi(x)|^{2\beta}+C\int_{t}^{T}|Y^{x}_s|^{2\beta}ds-2\beta\int_{t}^{T}|Y^{x}_{s}|^{2(\beta-1)}Y^{x}_{s}Z^{x}_{s}dW_s\nonumber\\
&&+2\beta\int_{t}^{T}|Y^{x}_{s}|^{2(\beta-1)}Y^{x}_{s}g(s,Y^{x}_s,Z^{x}_{s})d\overleftarrow{B}_s.
\end{eqnarray*}
Therefore, by Doob's maximal inequality
\begin{eqnarray*}
\E\left(\sup_{0\leq t\leq T}|Y^{x}_t|^{4\beta}\right)&\leq& \E|\xi(x)|^{4\beta}+C\int_{t}^{T}|Y^{x}_s|^{4\beta}ds+C\E\left[\sup_{0\leq t\leq }\left|\int_{t}^{T}|Y^{x}_{s}|^{2(\beta-1)}Y^{x}_{s}Z^{x}_{s}dW_s\right|^2\right]\nonumber\\
&&+C\E\left[\sup_{0\leq t\leq }\left|\int_{t}^{T}|Y^{x}_{s}|^{2(\beta-1)}Y^{x}_{s}g(s,Y^{x}_s,Z^{x}_{s})d\overleftarrow{B}_s\right|^2\right]\\
&\leq& \E|\xi(x)|^{4\beta}+C\int_{t}^{T}\E|Y^{x}_s|^{4\beta}ds+C\int_{t}^T\E(|Y^{x}_s|^{2(2\beta-1)}|Z^{x}_s|^{2})ds\\
&\leq& \E|\xi(x)|^{4\beta},
\end{eqnarray*}
which yields the result by $({\bf H}^{3'}_{\xi})$.
\end{proof}

{\bf Proof of Theorem 1.2}
By $({\bf H}^{3'}_{\xi})$, Theorem 2.1 and Lemma 2.2, the mappings $x\mapsto Y^x_t(\omega)$ are continuous injective for all $t\in[0, T ]$, a.s. With the help of Lemma 2.3, the proof of surjection of $x\mapsto Y^{x}_t (\omega)$ is just a repeat of (\cite{QZ}, p.13) and we therefore omit the details.
\section{Applications}
\setcounter{theorem}{0}
\setcounter{equation}{0}

In this section we consider Eq.$(\ref{FBSDE})$ and work on the framework of Pardoux-Peng \cite{PP}, assuming that
\begin{description}
\item $({\bf C}^1_{f})$ for every $s\in[0, T ],\ (x, y, z) \mapsto f(s, x, y, z)$ is of class $C^3$, the first order partial
derivatives in $y$ and $z$ are bounded on $[0, T ]\times\R\times\R\times\R$, as well as their derivatives of
order one and two with respect to $x, y, z$;
\item $({\bf C}^2_{f})$ for every $s\in[0, T ]$, the function $x\mapsto f (s, x, 0, 0)$ has polynomial growth at infinity
together with all partial derivatives up to order three;
\item $({\bf C}^3_{f})$ for every $s\in[0, T ]$ and $y, z\in\R$, the function $\mapsto f(s, x, y, z)$ is increasing (or
decreasing) in $x$;
\item $({\bf C}^4_{f})$ for some $C_1 > 0$ and $\varepsilon_1 > 0$, it holds that $y.f (s, x, y, z) > -C_1|z|^2$ for all $s\in[0, T ]$ and $|y|\leq \varepsilon_1,\; x, z\in\R$;
\item $({\bf C}^1_{g})$ $g(s,x,y)=a(s,x)y$ such that, for every $s\in[0,T], (x, y) \mapsto g(s, x, y)$ is of class $C^{3}$, the first order partial derivatives in $y$ is bounded.
\item $({\bf C}^1_{\sigma,b})$ $\sigma, b\in C^{3}_b(\R)$ have all bounded derivatives up to order three;
\item $({\bf C}^2_{\sigma,b})$ there are constants $c_1 > 0$ such that
\begin{eqnarray*}
|b(x)| + |\sigma(x)|\leq c_1|x|;
\end{eqnarray*}
\item $({\bf C}^1_{h})$ $h$ is of class $C^3$ with polygonal growth derivatives up to order three;
\item $({\bf C}^2_{h})$ $x\mapsto h(x)$ is increasing (or decreasing) and a homeomorphism on $\R$;
\item $({\bf C}^3_{h})$ there are constants $c_2, \gamma> 0$ such that
\begin{eqnarray*}
|h(x)| \geq c_2|x|^{\gamma}.
\end{eqnarray*}
\end{description}
Consider the following quasilinear parabolic stochastic partial differential equations:
\begin{eqnarray}
\left\{
\begin{array}{l}
\displaystyle{\frac{\partial u}{\partial t}(t,x)+\mathcal{L}u(t,x)+f(s,x,u(t,x),(\partial_x u.\sigma)(t,x))+a(s,x)u(t,x)\lozenge B_s}\\\\
u(T,x)=h(x)
\end{array}\right.
\label{SPDE}
\end{eqnarray}
where $u : [0, T ]\times\R\rightarrow\R,\; \mathcal{L} =\frac{1}{2}\sigma^2(x)\frac{\partial^2}{\partial x^{2}}+ b(x) \frac{\partial x}{\partial x}$ and $\lozenge$ denotes the Wick product, which indicates that the differential is to understand in Itô's sense.
Pardoux-Peng \cite{PP} proved the following result:
\begin{theorem}
Under the assumptions $({\bf C}^1_{f})$, $({\bf C}^2_{f})$, $({\bf C}^1_{g})$, $({\bf C}^1_{\sigma,b})$ and $({\bf C}^1_{h})$, for any $t\in[0, T ]$, let
$\{(Y^{t,x}_s , Z^{t,x}_s ),\ s\in[t, T ]\}$ be the solution to Eq. $(\ref{FBSDE})$, and define
\begin{eqnarray*}
u(t, x) = Y^{t,x}_t,
\end{eqnarray*}
then $u\in C^{1,2}([0, T ] \times\R)$ is the unique solution to Eq. $(\ref{SPDE})$.
\end{theorem}
We need the following lemma which is proved in \cite{QZ}.
\begin{lemma}
Assume that $({\bf C}^1_{\sigma,b})$ and $({\bf C}^2_{\sigma,b})$ hold. Then, for any $\beta<$ 0 there is a constant $C> 0$
such that
\begin{eqnarray*}
\E|X^{t,x}_{s}|^{2\beta}\leq C|x|^{2\beta},\;  t\in[0, T ], s\in[t, T ], |x| > 1.
\end{eqnarray*}
\end{lemma}

Next, applying the well known comparison theorem about the forward stochastic differential equation
(see \cite{QZ}), and the above lemma as well as Theorems 2.1, 1.2 and 3.1, we can prove that
\begin{proposition}
Under the beginning assumptions of this section, for any $t\in[0, T ]$, the
mappings $x\mapsto Y^{t,x}_s(\omega)$ are homeomorphisms on $\R$ for all $s\in[t, T ]$ a.s. In particular, the
unique solution to Eq. $(\ref{SPDE})\;  x\mapsto u(t, x)$ is a homeomorphism on $\R$.
\end{proposition}
\begin{remark}
Originally we intended to treat the problem in the present paper in general case i.e $g$ nonlinear, in the hope of obtaining the homeomorphic property.But we have revised our ambition to fuck. Indeed, one knows that if $g$ is not linear, the backward SDE $(\ref{SDE})$ has not an explicit solution, which does not provide proof of surjective.
\end{remark}


\begin{thebibliography}{99}


\bibitem{BM1} Buckdahn, R.; Ma, J., Stochastic viscosity solutions for nonlinear stochastic partial differential equations. I. {\it Stochastic Process. Appl.} {\bf 93} (2001), no. 2, 181-204.

\bibitem{BM2} Buckdahn, R.; Ma,J., Stochastic viscosity solutions for nonlinear stochastic partial differential equations. II. {\it Stochastic Process. Appl.} {\bf 93} (2001), no. 2, 205-228.

\bibitem{Kal} El Karoui, N.; Peng, S.; Quenez, M. C., Backward stochastic differential equations in finance. {\it Math. Finance} {\bf 7} (1997), no. 1, 1-71.

\bibitem{K} H. Kunita, Stochastic Differential Equations and Stochastic Flows of Diffeomorphisms, {\it in: Lect. Notes in Math.},
vol. {\bf 1097}, Springer-Verlag, 1984, pp. 143-303.

\bibitem{P} P. Protter, Stochastic Integration and Differential Equation, second ed., Springer-Verlag, Berlin, 2004.

\bibitem{PP}  Pardoux, E. and Peng, S., Backward doubly stochastic differential equations and systems of quasilnear SPDEs {\it Probab. Theory Related Fields.} {\bf
98} (1994), no. 2, 209-227.


\bibitem{QZ} Qiao, H.; Zhang, X. Homeomorphism of solutions to backward SDEs and applications. {\it Stochastic Process. Appl.} {\bf 117} (2007), no. 3, 399-408.

\bibitem {YO} Yamada, T; Y. Ogura, Y., On the strong comparison theorems for solutions of stochastic differential equations, {\it Z. W.
verw. Gebiete} {\bf 56} (1981) 3-19.

%\bibitem{Z} Zhang Jianfeng, A numerical scheme for BSDEs. {\it  Ann. Appl. Probab.} {\bf 14} (2004), no. 1, 459-488.

\bibitem {RZ} Ren, J.; Zhang, X, Stochastic flows for SDEs with non-Lipschitz coefficients, {\it Bull. Sci. Math.} {\bf 127} (2003) 739-754.

%\bibitem{Yal} Yufen Shi, Yanling Gu and Kai Liu. Comparison theorem of backward doubly stochastic differential equations and application. {\sl Stoch. Anal. Appl. } {\bf 23} (2005), no.1, $97-110$.
\end{thebibliography}
\end{document}